\documentclass[11pt]{amsart}
\usepackage{amsfonts}
\usepackage{amssymb,latexsym}
\usepackage{enumerate}
\usepackage{mathrsfs}
\usepackage{amsmath}
\usepackage{amsthm}
\usepackage{hyperref}
\newtheorem{thm}{Theorem}[section]
\newtheorem{prop}[thm]{Proposition}

\numberwithin{equation}{section}

\newcommand{\X}{$X$}
\newcommand{\bp}{\overline{\partial}}

\begin{document}

\title{\textsc{A note of semistable  Higgs bundles over compact K\"{a}hler manifolds}}
\author{Yanci Nie and Xi Zhang}
\address{Yanci Nie\\School of Mathematical Sciences\\
University of Science and Technology of China\\
Hefei, 230026\\ } \email{nieyanci@mail.ustc.edu.cn}
\address{Xi Zhang\\Key Laboratory of Wu Wen-Tsun Mathematics\\ Chinese Academy of Sciences\\School of Mathematical Sciences\\
University of Science and Technology of China\\
Hefei, 230026,P.R. China\\ } \email{mathzx@ustc.edu.cn}
\thanks{The authors were supported in part by NSF in China,  No.11131007.}
\maketitle
\begin{abstract}
In this note, by using the Yang-Mills-Higgs flow, we show that  semistable Higgs  bundles with vanishing the first and second Chern numbers over  compact K\"{a}her manifolds must admit a filtration whose quotients are Hermitian flat Higgs bundles.
\end{abstract}

\section{Introduction}

 Let $(X, \omega )$ be an $n$-dimensional compact K\"ahler manifold, where $\omega$ is the K\"ahler form. A Higgs bundle $(E , \overline{\partial }_{E}, \phi )$ over $X$ is a holomorphic bundle $(E , \overline{\partial }_{E})$ coupled with a Higgs field $\phi \in \Omega_X^{1,0}(End(E))$
such that $\overline{\partial}_{E}\phi =0$ and $\phi \wedge \phi =0$.
Higgs bundle was introduced by Hitchin (\cite{hitchin1987self}) in his study of the self-duality
equations on a Riemann surface. It has  rich structures and plays a role in many areas including gauge theory, K\"ahler and hyperk\"ahler geometry, group representations and nonabelian Hodge theory (\cite{simpson1988constructing}, \cite{simpson1992higgs}).
Given a coherent sheaf $\mathcal{F}$ on $X$, the $\omega $-slope $\mu_{\omega }(\mathcal{F})$ of $\mathcal{F}$ is defined by
$$\mu_{\omega}(\mathcal{F})=\frac{\mbox{deg}_{\omega}(\mathcal{F})}{\mbox{rank} F}=\frac{1}{\mbox{rank} \mathcal{F}}\int_{X}c_1(\mathcal{F})\wedge \frac{\omega^{n-1}}{(n-1)!}.$$
 A Higgs bundle $(E, \overline{\partial }_{E}, \phi)$ over $X$ is $\omega$-stable(resp. $\omega$-semistable) if $\mu_{\omega}(\mathcal{F})<\mu_{\omega}(E)(\mbox{resp}.\  \mu_{\omega}(\mathcal{F})\leq\mu_{\omega}(E))$ for every proper  $\phi$-invariant coherent subsheaf $\mathcal{F}$ of $E$.
 When the K\"ahler form $\omega$ is understood, we shall sometimes refer to
$(E, \overline{\partial }_{E}, \phi)$ simply stable or semistable, and we will also omit subscripts and write $\mu (E)$.

Given a Hermitian metric $H$ on a Higgs bundle $(E, \overline{\partial }_{E}, \phi  )$, we can consider the  Hitchin-Simpson connection (\cite{simpson1988constructing,simpson1992higgs})
 \begin{eqnarray*}D_{H, \overline{\partial }_{E}, \phi }= D_{H, \overline{\partial }_{E}} + \phi +\phi ^{\ast H},\end{eqnarray*}
where $D_{H, \overline{\partial }_{E}}$ is the Chern connection with respect to the metric $H$ and the holomorphic structure $\overline{\partial }_{E}$, and $\phi ^{\ast H}$ is the adjoint of $\phi $ with respect to the metric $H$.
The curvature of the  Hitchin-Simpson connection is
\begin{eqnarray*}
F_{H, \overline{\partial }_{E}, \phi}=F_{H} +[\phi , \phi ^{\ast H}] +D_{H, \overline{\partial }_{E}}^{1, 0}\phi + \overline{\partial }_{E} \phi^{\ast H},
\end{eqnarray*}
where $F_{H}$ is the curvature of the Chern connection $D_{H, \overline{\partial }_{E}}$. We say a Higgs bundle $(E, \overline{\partial }_{E}, \phi  )$ is Hermitian flat if there exists a Hermitian metric $H$ such that its Hitchin-Simpson curvature vanishes.
A Hermitian metric $H$ on Higgs bundle $(E, \overline{\partial }_{E}, \phi  )$ is said to be
 Hermitian-Einstein  if the
curvature $F_{H, \overline{\partial }_{E}, \phi}$ of the Hitchin-Simpson connection
  satisfies the Einstein condition, i.e
\begin{eqnarray*}
\sqrt{-1}\Lambda_{\omega} (F_{H} +[\phi , \phi ^{\ast H}])
=\lambda Id_{E},
\end{eqnarray*}
where  $\Lambda_{\omega }$ denotes the contraction of differential
forms by K\"ahler form $\omega $, and the real constant $\lambda $ is
given by $\lambda  =\frac{2\pi}{Vol(X)} \mu (E)$.
In \cite{hitchin1987self} and \cite{simpson1988constructing}, it is proved that a Higgs bundle admits
Hermitian-Einstein metric iff it's Higgs poly-stable. This is a
Higgs bundle version of the Donaldson-Uhlenbeck-Yau theorem. By Chern-Weil theory, we have
\begin{equation}\label{111}
\begin{split}
& \int_X |F_{H}+[\phi,\phi ^{*H}]|_{H }^{2}+2|\partial_{H}\phi |_{H }^{2} \frac{\omega^{n}}{n!}\\
=&\int_X |\sqrt{-1}\Lambda_{\omega}(F_{H}+[\phi ,\phi ^{*H }])-\lambda Id_{E}|^{2}_{H}\frac{\omega^{n}}{n!}
\\&+4\pi^{2}\int_{X}(2c_{2}(E)-c_{1}^{2}(E))\wedge\frac{\omega^{n-2}}{(n-2)!}+\lambda^{2}rk(E)Vol(X).\\
\end{split}
\end{equation}
From the above formula, it is easy to see that a stable Higgs bundle $(E, \overline{\partial }_{E}, \phi  )$ with $c_{1}(E)\cdot [\omega]^{n-1}=0$, $ch_{2}(E)\cdot [\omega]^{n-2}=0$ must be Hermitian flat.

We say a Higgs bundle $(E, \overline{\partial }_{E}, \phi  )$ admits an approximate Hermitian-Einstein structure if for every positive $\epsilon
$, there is a Hermitian metric $H_{\epsilon}$ such that
\begin{eqnarray*}
\max _{X} |\sqrt{-1}\Lambda_{\omega }(F_{H_{\epsilon}}+[\phi , \phi^{\ast H_{\epsilon}}])-\lambda Id_{E}|_{H_{\epsilon}}<\epsilon .
\end{eqnarray*}
In \cite{bruzzo2007metrics}, Bruzzo and Gra\~{n}a Otero proved that a Higgs bundle admitting  an approximate Hermitian-Einstein structure must be semi-stable. The converse implication was proved in \cite{car} when $\dim{X}=1$, and in \cite{li2012existence} for arbitrary dimension. So we know that, in Higgs bundles, admitting an approximate Hermitian-Einstein structure and the semi-stability are equivalent.

In the projective case, Simpson (Theorem 2 in \cite{simpson1992higgs}) proved that a semi-stable Higgs bundle $(E, \overline{\partial }_{E}, \phi  )$  with $c_{1}(E)\cdot [\omega]^{n-1}=0$, $ch_{2}(E)\cdot [\omega]^{n-2}=0$ must admit a filtration whose quotients are Hermitian flat Higgs bundles (rather than stable Higgs sheaves). In this paper, we generalize the above result to K\"ahler case. Simpson's proof essentially depends on Mehta and Ramanathan's result (Theorem 4.3 in \cite{mehta1984restriction}) that the restriction of a semistable holomorphic vector bundle to general hyperplane sections of certain arbitrarily high degrees is semistable. Since Mehta and Ramanathan's result is valid only when $X$ is a smooth projective variety, Simpson's argument can not be generalized to general K\"ahler  case directly. In this paper, we use the Yang-Mills-Higgs flow to study this problem. We prove that:
\begin{thm}\label{5}
 Let $(X,\omega)$ be a compact K\"ahler manifold and $\mathfrak{E}=(E,\overline{\partial }_{E}, \phi)$ be a Higgs bundle over $X$. $\mathfrak{E}$ is Higgs semi-stable with $c_{1}(E)\cdot [\omega]^{n-1}=0$, $ch_{2}(E)\cdot [\omega]^{n-2}=0$ if and only if it has a filtration in Higgs sub-bundles such that all quotients are Hermitian flat Higgs bundles.
\end{thm}

\medskip

This paper is organized as follows. In section 2, we recall some basic notions for Higgs bundles including  some basic results of the Donaldson heat flow and the Yang-Mills-Higgs flow.
In section 3, we give a proof of Theorem \ref{5}.
\medskip

\section{Preliminaries}

Let $(X, \omega )$ be an $n$-dimensional compact K\"ahler manifold, where $\omega$ is the K\"ahler form.
A Higgs vector bundle $(E,\overline{\partial}_E,\phi)$ is called Higgs approximate Hermitian flat, if for every $\epsilon>0$, there exists a Hermitian metric $H_{\epsilon}$, such that \begin{eqnarray*}\sup\limits_{X}|F_{H_{\epsilon}, \overline{\partial }_{E}, \phi}|_{H_{\epsilon}}<\epsilon.\end{eqnarray*}

Let $S$ be a $\phi$-invariant holomorphic  sub-bundle of  $(E,\overline{\partial}_E,\phi)$. Consider the exact sequence of Higgs bundles:
$$0\longrightarrow (S,\phi_S)\stackrel{ i}{\longrightarrow}(E,\phi)\stackrel{p}{\longrightarrow}(Q,\phi_Q)\longrightarrow0,$$
where $Q$ is the quotient $E/S$, $\phi_S$ and $\phi_Q$ are the induced Higgs fields. Given a Hermitian metric on $E$, we have the following exact sequence:

$$0\longleftarrow(S,\phi_S)\stackrel{\pi^{*H}}{\longleftarrow}(E,\phi)\stackrel{p^{*H}}{\longleftarrow}(Q,\phi_Q)
\longleftarrow0.$$
From the above, we have the bundle isomorphism on \X:
$$f_H= i\oplus p^{*H}:S\oplus Q\rightarrow E.$$
Then the pull-back metric, holomorphic structure and Higgs field on $S\oplus Q$ are
$$f^{\ast}_H(H)=\left (
\begin{matrix}
H_{S} &  0 \\
0   & H_{Q}\\
\end{matrix}
\right ),\ \ \ \ f_{H}^{\ast }(\overline{\partial }_{E})=\left (
\begin{matrix}
\overline{\partial}_{S} &  \gamma \\
0   & \overline{\partial}_{Q}\\
\end{matrix}
\right ),$$

$$f^{\ast}_H(\phi)=\left(
\begin{matrix}
\phi_S & \zeta \\
0 & \phi_Q \\
\end{matrix}\right),
$$
where $\gamma (t)\in \Omega_X^{0,1}(Q^*\otimes S)$ is the second fundamental form and $\zeta\in\Omega_X^{1,0}(Q^*\otimes S)$.
We have  the following Gauss-Codazzi equation for the Hitchin-Simpson's curvature:
\begin{equation}\label{112}
\begin{split}
f^*_H(F_{H, \overline{\partial }_{E}, \phi})=
&f^{\ast}_H(F_H)+f^{\ast}_H([\phi,\phi^{\ast}])+f^{\ast}_H(\partial_H(\phi))+f^{\ast}_H(\overline{\partial}_E(\phi^*))
\\=&\left (
\begin{matrix}
F_{H_{S}}-\gamma \wedge \gamma^{\ast}  &  D^{1,0}_{S\otimes Q^{\ast}}\gamma \\
-D^{0,1}_{S^{\ast}\otimes Q}\gamma^{\ast}   & F_{H_{Q}}-\gamma^{\ast} \wedge \gamma \\
\end{matrix}
\right )
\\+&\left(
\begin{matrix}
[\phi_S,\phi_S^{\ast}]+\zeta\wedge\zeta^{\ast} & \zeta\wedge\phi_Q^{\ast}+\phi_S^{\ast}\wedge\zeta \\
\zeta^{\ast}\wedge\phi_S+\phi_Q\wedge\zeta^{\ast}  & [\phi_Q,\phi_Q^{\ast}]+\zeta^{\ast}\wedge\zeta \\ \end{matrix}\right)
\\+&\left(\begin{matrix}
\partial_{H_{S}}\phi_{S}-\zeta \wedge \gamma^{\ast }  &  D^{1,0}_{Q^{\ast}\otimes S} \zeta \\
-\gamma^{\ast }\wedge \phi_{S}-\phi_{Q}\wedge \gamma^{\ast}   & \partial_{H_{Q}}\phi_{Q}-\gamma^{\ast }\wedge \zeta \\
\end{matrix}\right)
\\+&\left (
\begin{matrix}
\overline{\partial}_{S}\phi_{S}^{\ast}+\gamma \wedge \zeta^{\ast }  &  \gamma \wedge \phi_{Q}^{\ast }+\phi_{S}^{\ast }\wedge \gamma  \\
\overline{\partial }_{S^{\ast}\otimes Q}\zeta^{\ast}   & \overline{\partial}_{Q}\phi_{Q}^{\ast }+\zeta^{\ast }\wedge \gamma \\
\end{matrix}
\right ).
\end{split}
\end{equation}
\begin{prop}\label{115}
Let $0\rightarrow (S,\phi_S){\rightarrow}(E,\phi){\rightarrow}(Q,\phi_Q)\rightarrow0$ be an exact sequence of Higgs bundles on \X. If $(S,\phi_S)$ and $(Q,\phi_Q)$ are Higgs approximate Hermitian flat, then $(E,\phi)$ is Higgs approximate Hermitian flat.
\end{prop}
\begin{proof}
Let $f:S\oplus Q\stackrel{\sim}{\rightarrow }E$ be a $C^{\infty}$ bundle isomorphism. For every $\epsilon>0$, let $H_{S,\epsilon}$ and $H_{Q,\epsilon}$ be Hermitian metrics on $S$ and $Q$ respectively, such that $\sup\limits_{X}|F_{H_{S,\epsilon},\overline{\partial}_S,\phi_S}|<\epsilon$ and $\sup\limits_{X}|F_{H_{Q,\epsilon},\overline{\partial}_Q,\phi_Q}|<\epsilon$. For small $\rho >0$, we set a Hermitian metric on $E$ by $H_{\epsilon,\rho}=(f^{-1})^{\ast}\widetilde{H}_{\epsilon,\rho},$ where
\begin{eqnarray*}
\widetilde{H}_{\epsilon,\rho}=\left(\begin{matrix} H_{S,\epsilon} & \\ & \frac{1}{\rho^2}H_{Q,\epsilon}
\end{matrix}\right).
\end{eqnarray*}
The pull-back holomorphic structure and Higgs field on $S\oplus Q$ can be written as:
$$f^{\ast}(\overline{\partial}_E)=\left(
\begin{matrix}
\overline{\partial}_S & \gamma\\
0  &  \overline{\partial}_Q
\end{matrix}
\right),\ \ \ f^{\ast}(\phi )=\left(\begin{matrix}
\phi_S & \zeta\\
0 & \phi_Q
\end{matrix}\right).$$
We set $\gamma^*_{\rho}=\gamma^{*\widetilde{H}_{\epsilon , \rho}}$ and $\ \zeta_{\rho}^*=\zeta^{*\widetilde{H}_{\epsilon , \rho}}$. A simple calculation implies $\gamma_{\rho}^*=\rho^2\gamma_1^*$ and $\zeta_{\rho}^*=\rho^2\zeta_1^*.$
Then
\begin{equation*}
\begin{split}
f^{\ast}(F_{H_{\epsilon , \rho},\overline{\partial}_E,\phi})=&\left(\begin{matrix}
F_{H_{S,\epsilon,},\overline{\partial}_S,\phi_S} & 0\\
0 & F_{H_{Q,\epsilon,},\overline{\partial}_Q,\phi_Q}
\end{matrix}
\right)
\\+&\left(
\begin{matrix}
(\zeta-\gamma)\wedge(\zeta_{\rho}^*+\gamma_{\rho}^*) &0
\\0 & (\zeta_{\rho}^*+\gamma_{\rho}^*)\wedge(\zeta-\gamma)
\end{matrix}
\right)
\\+&\left(
\begin{matrix}
0 & D^{1,0}_{Q^*\otimes S}(\gamma+\zeta)+
\\ &(\gamma+\zeta)\wedge\phi_Q^*+\phi_S^*\wedge(\gamma+\zeta)
\\D^{0,1}_{S^*\otimes Q}(\zeta_{\rho}^*-\gamma_{\rho}^*)+ & \\ (\zeta_{\rho}^*-\gamma_{\rho}^*)\wedge\phi_S+
\phi_Q\wedge(\zeta_{\rho}^*-\gamma_{\rho}^*) & 0
\end{matrix}
\right)
\\=:&A_{\epsilon}+B_{\rho,\epsilon}+C_{\rho,\epsilon} ,
\end{split}
\end{equation*}
and
\begin{eqnarray*}\begin{array}{lll}|F_{H_{\epsilon ,\rho},\overline{\partial}_E,\phi}|^2_{H_{\epsilon,\rho}}
&=&|f^{\ast}(F_{H_{\epsilon , \rho},\overline{\partial}_E,\phi})|^2_{\widetilde{H}_{\epsilon , \rho}}\\
&\leq &|A_{\epsilon}|^2_{\widetilde{H}_{\epsilon , \rho}}+|B_{\rho,\epsilon}|^2_{\widetilde{H}_{\epsilon \rho}}+|C_{\rho,\epsilon}|^2_{\widetilde{H}_{\epsilon \rho}}\\
&=&|A_{\epsilon}|^2_{\widetilde{H}_{\epsilon , 1}}+\rho^4|B_{1,\epsilon}|^2_{\widetilde{H}_{\epsilon , 1}}+\rho^2|C_{1,\epsilon}|^2_{\widetilde{H}_{\epsilon , 1}}.\\
\end{array}
\end{eqnarray*}
For every $\epsilon>0$, we can choose $\rho$ small enough, such that $\sup\limits_{X}|F_{H_{\epsilon , \rho},\overline{\partial}_E,\phi}|_{H_{\epsilon , \rho}}<3\epsilon$.
\end{proof}

\medskip

The following proposition can be seen as  a conclusion of the Proposition 3.10. in \cite{bruzzo2007metrics}. Here,  we  give a proof in detail for reader's convenience.

\begin{prop}\label{113}
Let $(E,\overline{\partial}_E,\phi)$ be a Higgs approximate Hermitian flat vector bundle over a compact K\"ahler manifold $(X,\omega)$. If $s$ is a nontrivial $\phi$-invariant holomorphic section of $E$, then $s$ has no zeros.
\end{prop}
\begin{proof}
Since $s$ is a $\phi$-invariant, so we have $\phi (s)=\eta \otimes s,$
where $\eta $ is a $(1, 0)$-form. Let us choose a local complex coordinates $\{z^{i}\}_{i=1}^{n}$. We can write $\phi,\ \eta$ as $\phi=\phi_i dz^i$, $\eta=\eta_idz^i$, where $\phi_i$ are holomorphic endomorphisms of $E$ satisfying $\phi_i(s)=\eta_i s$. Set $V=V^{i}\frac{\partial }{\partial z^{i}}$. Given a Hermitian metric $H$ on $E$, we have
\begin{eqnarray}
\begin{array}{lll}
\sqrt{-1}H([\phi , \phi^{\ast H}] s, s )(V, \overline{V})&=&\sum\limits_{i,j=1}^{n}V^{i}\overline{V^{j}}H(\{\phi_{i}(\phi_{j})^{\ast H}-(\phi_{j})^{\ast H}\phi_{i}\}s, s)\\
&=&\sum\limits_{i,j=1}^{n}V^{i}\overline{V^{j}}\{H((\phi_{j})^{\ast H}s, (\phi_{i})^{\ast H}s) -H(\phi_{i}s, \phi_{j} s)\}\\
&=&\sum\limits_{i,j=1}^{n}V^{i}\overline{V^{j}}\{ H(\overline{\eta_j}s+G_{j}, \overline{\eta_i}s+G_{i}    )\\
& &-\eta_{i}\overline{\eta_{j}}H(s, s)\}\\
&=&\sum\limits_{i,j=1}^{n}H(\overline{V^{j}}G_{j}, \overline{V^{i}}G_{i})\geq 0,\\
\end{array}
\end{eqnarray}
where $G_{j}=(\phi_{j})^{\ast H}s-\overline{\eta_j}s$. So, we know that $\sqrt{-1}H([\phi , \phi^{\ast H}] s, s )$ is a nonnegative $(1, 1)$-form, i.e.
\begin{eqnarray}\label{nonn}
\sqrt{-1}H([\phi , \phi^{\ast H}] s, s )\geq 0.
\end{eqnarray}

For $\forall \epsilon >0$, there exists a hermitian metric $H_{\epsilon}$ such that $$\sup\limits_X |F_{H_{\epsilon}, \overline{\partial }_{E}, \phi}|<\epsilon.$$ Consider the $(1,1)$ current
$$T_{\epsilon}=\frac{\sqrt{-1}}{2\pi }\partial \overline{\partial }\log H_{\epsilon }(s,s).$$
By (\ref{nonn}) and direct calculation, we have
\begin{eqnarray}
\begin{array}{lll}
T_{\epsilon}&\geq &-\frac{H_{\epsilon}(\frac{\sqrt{-1}}{2\pi}F_{H_{\epsilon}}(s),s)}{H_{\epsilon}(s,s)}\\
&\geq & -\frac{H_{\epsilon}(\frac{\sqrt{-1}}{2\pi}\{F_{H_{\epsilon}}+[\phi , \phi ^{\ast H_{\epsilon}}]\}(s),s)}{H_{\epsilon}(s,s)}\\
&\geq &-\delta(\epsilon)\omega .\\
\end{array}\end{eqnarray}
in current sense, where  $\delta(\epsilon)\rightarrow 0$ as $\epsilon\rightarrow 0$.
Since $\omega $ is K\"ahler, we have
$$0\leq\int_X \left(T_{\epsilon}+\delta(\epsilon)\omega\right)\wedge \omega^{n-1}=\delta(\epsilon)\int_X \omega^{n}.$$
This implies $T_{\epsilon} \rightharpoonup 0$ (subsequently) weakly in current sense.
If $s(x)=0$ at some point $x\in X$, the Lelong number $\nu (T_{\epsilon }, x)\geq 1$. Thus by \cite{Siu}, the Lelong number of weak limit $T=\lim T_{\epsilon}$ satisfies  $\nu (T, x)\geq 1$. This gives a contradiction. So the nontrivial $\phi$-invariant section $s$ has no zeros.
\end{proof}
\medskip
Given a Hermtian metric $H_{0}$ on Higgs bundle $(E, \overline{\partial }_{E}, \phi )$ over K\"ahler manifold $(X, \omega )$, one can evolve the Hermitian metric by the following Donaldson's heat flow:
\begin{equation}\label{1}
H^{-1}\frac{\partial H(t)}{\partial t}=-2(\sqrt{-1}\Lambda_\omega(F_{H(t)}+[\phi,\phi^{*H(t)}])-\lambda Id_E).
\end{equation}
In \cite{simpson1988constructing}, Simpson proved the global existence of the above heat flow and showed the convergence under the condition that the Higgs bundle is stable. Without the assumption of the stability, the above flow (\ref{1}) may not converge at infinity. When the Higgs bundle $(E, \overline{\partial }_{E}, \phi)$ is semistable, in \cite{li2012existence}, the authors studied the asymptotic properties of the above flow and proved the existence of approximate Hermitian-Einstein metric structure. In fact, they proved that:

\begin{prop}\mbox{\rm{((3.14) in \cite{li2012existence})}}\label{6}
Let $H(t)$ be a solution of the Donaldson's heat flow (\ref{1}) on Higgs bundle $(E, \overline{\partial }_{E}, \phi)$.  If $(E, \overline{\partial }_{E}, \phi)$ is Higgs semi-stable, then
\begin{equation}\label{2}
\int_{X}|\sqrt{-1}\Lambda_\omega(F_{H(t)}+[\phi,\phi^{*H(t)}])-\lambda Id_E|^{2}_{H(t)}\frac{\omega^{n}}{n!}
\rightarrow 0
\end{equation}
as $t\rightarrow +\infty$.
\end{prop}

Fixing a Hermitian metric $H_{0}$ on the bundle $E$, we consider the Higgs structures from the Gauge theory. Let
\begin{eqnarray}
\mathcal{B}_{E, H_{0}} =\left\{(A,\phi)\in\mathcal{A}^{1,1}_{H_0}\times \Omega^{1,0}(End (E))\Big|\bar{\partial}_A(\phi)=0, \phi\wedge\phi=0\right\}
\end{eqnarray}
where ${ \mathcal{A}}^{1,1}_{H_{0}}$ denotes the space of unitary integrable connections of $E$ compatible with the metric $H_{0}$. It is well known that given a unitary integrable connection
$A\in \mathcal{A}^{1,1}_{H_0}$  (i.e. whose curvature $F_{A}$ is of type
$(1,1)$ ), $D_{A}^{(0,1)}=\overline{\partial }_{A}$ defines a
holomorphic structure on $E$. A pair $(A, \phi )\in \mathcal{B}_{E, H_{0}}$ is called a Higgs pair and it determines one Higgs structure on $E$, i.e. $(E, D_{A}^{0, 1}, \phi )$ is a Higgs bundle.

The Yang-Mills-Higgs functional is defined on the space of Higgs pairs $\mathcal{B}_{E, H_{0}}$
by
\begin{equation}
 \mbox{\rm{YMH}}(A,\phi)=\int_X (|F_A+[\phi,\phi^{*H_0}]|^{2}+2| \partial_A\phi|^{2})d\,V_g .
\end{equation}
\\
We consider the gradient heat flow of the above functional, i.e. the following Yang-Mills-Higgs flow
\begin{equation}
\begin{cases}\label{3}
\frac{\partial A}{\partial t}=-D^*_A F_A-\sqrt{-1}(\partial_A\Lambda_\omega-\overline{\partial}_A\Lambda_\omega)[\phi,\phi^{*H_0}],
\\
\frac{\partial \phi}{\partial t}=-[\sqrt{-1}\Lambda_\omega(F_A+[\phi,\phi^{*H_0}]),\phi],
\\
A|_{t=0}=A_0,\ \ \ \ \phi|_{t=0}=\phi_0.
\end{cases}
\end{equation}
The global existence and
uniqueness of the solution for the above gradient flow had been proved in \cite{li2011gradient} (also in \cite{Wi}). In fact,
given any initial  Higgs pair $(A_0,\phi_0)$, the Yang-Mills-Higgs flow has a unique solution $(A(t),\phi(t))$ in the complex gauge orbit of $(A_0,\phi_0)$, i.e. $(A(t),\phi(t))=g(t)(A_0,\phi_0)$, where $g(t)^{*H_0}g(t)=H_0^{-1} H(t),\, g(t)\in \mathcal{G}^{\mathcal{C}}$ and $ H(t)$ is the solution of \rm{(\ref{1})} on Higgs bundle $(E, \overline{\partial }_{A_{0}}, \phi_{0})$ with initial data $H_{0}$. Here $\mathcal{G}^{\mathcal{C}}$ is the complex gauge group, and it acts on the space of Higgs pair $\mathcal{B}_{(E,H_0)}$ as following: let
$\sigma \in { \mathcal{G}}^{C}$,
\begin{eqnarray}
\overline{\partial }_{\sigma(A)}=\sigma \circ \overline{\partial
}_{A}\circ \sigma^{-1}, \quad \partial _{\sigma (A)}=(\sigma^{\ast
H_{0}})^{-1} \circ \partial _{A}\circ \sigma^{\ast H_{0}};
\end{eqnarray}
\begin{eqnarray}
\sigma (\phi )=\sigma \circ \phi \circ \sigma^{-1}.
\end{eqnarray}
By the definitions, the Yang-Mills-Higgs flow (\ref{3}) is evolve Higgs pairs on a fixed Hermitian bundle $(E, H_{0})$ and the Donaldson's heat flow (\ref{1}) is evolve  Hermitian metrics on a fixed Higgs bundle $(E, \overline{\partial }_{E}, \phi )$. From the above, we know that these two heat flows are equivalent modulating a sequence of complex gauge transformations.

Let $H(t)$ be the solution of the Donaldson's heat flow (\ref{1}) on the Higgs bundle $(E, \overline{\partial }_{E}, \phi_{0})$ with initial metric $H_{0}$, and $(A(t), \phi (t))$ be the solution of the Yang-Mills-Higgs flow (\ref{3})  on the Hermitian bundle $(E, H_{0})$ with initial data $(A_{0}, \phi_{0})$, where
$A_{0}=D_{(E, \overline{\partial }_{E}, H_{0})}$. It is easy to check the following relations:
\begin{equation}
\begin{split}
\\&\overline{\partial}_{A(t)}\phi^{*H_0}=g(t)\circ\overline{\partial}_{E}\phi_{0}^{*H(t)}\circ g(t)^{-1},
\\&\partial_{A(t)}\phi(t)=g(t)\circ\partial_{H(t)}\phi_{0}\circ g(t)^{-1},
\\&F_{A(t)}+[\phi(t),\phi(t)^{*H_0}]=g(t)\circ( F_{H(t)}+[\phi_{0},\phi_{0}^{*H(t)}])\circ g(t)^{-1},
\end{split}
\end{equation}
and
\begin{equation}\label{4}
\begin{split}
&| \partial_{A(t)}\phi(t)|_{H_{0}}^{2}=|\bp_{A(t)}\phi^{*H_0}|_{H_{0}}^2=|\partial_{H(t)}\phi_{0}|_{H(t)}^{2}=|\bp_{E}\phi_{0}^{*H(t)}|_{H(t)}^2,
\\&|(F_{A(t)}+[\phi(t),\phi(t)^{*H_0}])|_{H_{0}}^{2}=|F_{H(t)}+[\phi_{0},\phi_{0}^{*H(t)}]|_{H(t)}^{2},
\\&|\sqrt{-1}\Lambda_{\omega}(F_{A(t)}+[\phi(t),\phi(t)^{*H_0}])|_{H_0}^{2}
=|\sqrt{-1}\Lambda_{\omega}(F_{H(t)}+[\phi_{0},\phi_{0}^{*H(t)}])|_{H(t)}^{2}.
\end{split}
\end{equation}

\medskip

Along the Yang-Mills-Higgs flow, we have
the following inequality (\cite[p.~1384]{li2011gradient})
\begin{eqnarray}\label{8}
\Big(\triangle_g-\frac{\partial}{\partial t}\Big)| \phi |_{H_0}^2\geq2| \nabla_A \phi|_{H_0}^2+C_1(|\phi|_{H_0}^2+1)^2-C_2(|\phi|_{H_0}^2+1),
\end{eqnarray}
and the uniform $C^0$ bound of $\phi$ (Lemma $2.3$ in \cite{li2011gradient})
\begin{eqnarray}\label{9}
\sup \limits_{X}|\phi(t)|_{H_0}^{2}\leq \sup\limits_{X}|\phi_{0}|_{H_0}^{2}+C_{3},
\end{eqnarray}
where constants $C_{i}$ depend only on the geometry of $(X,\omega)$ and the initial data $(A_0,\phi_0)$.
We also have the  Bochner type inequality ((2.12) in \cite{li2011gradient})
\begin{equation}\label{11}
\begin{split}
&\Big(\triangle_g-\frac{\partial}{\partial t}\Big)(|F_{A(t)}+[\phi(t),\phi(t)^{*H_0}]|_{H_0}^{2}+2|\partial_{A(t)}\phi(t)|_{H_{0}}^{2})
\\&-2|\nabla_{A(t)}(|F_{A(t)}+[\phi(t),\phi(t)^{*H_0}]|_{H_{0}}^{2})-4|\nabla_{A(t)}(\partial_{A(t)}\phi(t))|_{H_{0}}^{2}
\\ \geq&-C(n)(|F_{A(t)}+[\phi(t),\phi(t)^{*H_0}]|_{H_{0}}+|\nabla_{A(t)}\phi(t)|_g+|\phi(t)|_{H_{0}}^{2}+|Rm|_g)
\\&(|F_{A(t)}+[\phi(t),\phi(t)^{*H_0}]|_{H_{0}}^{2}+2|\partial_{A(t)}\phi(t)|_{H_{0}}^{2}),
\end{split}
\end{equation}
where $C(n)$ is a constant depending only on the dimension of $X$. In the following, we denote
\begin{eqnarray*}
e(A(t), \phi(t))=|F_{A(t)}+[\phi(t),\phi(t)^{*H_0}]|_{H_0}^{2}+2|\partial_{A(t)}\phi(t)|_{H_{0}}^{2},
\end{eqnarray*}
and
\begin{eqnarray*}
P_r(x_0 , t_{0})=B_r(x_0)\times[t_0-r^2,t_0+r^2]
\end{eqnarray*}
for any fixed point $(x_{0}, t_{0})\in X\times \mathbb{R}_{+}$.
$\epsilon$-regularity theorem  is important in studying the asymptotic behavior of the Yang-Mills flow ( \cite{hong2004asymptotical}). For the Yang-Mills-Higgs flow (\ref{3}), we have the following $\epsilon$-regularity theorem. (Theorem 3.1 in \cite{li2011gradient} for K\"ahler surface case, Theorem 2.6. in \cite{li2014} and Theorem 4 in for higher dimensional case. )

\medskip

\begin{prop}\label{21}\mbox{\rm{($\epsilon $-regularity theorem)}}
 Let $(A(t),\phi (t))$ be a smooth solution of (\ref{3}) over an $n$-dimensional compact K\"ahler manifold $(X,\omega)$ with initial value $(A_0,\phi_0)$. There exist positive constants $\epsilon_0,\delta_0<1/4$, such that, if for some $0<R<\min\{i_{X}, \frac{\sqrt{t_0}}{2}\}$, the inequality
\begin{equation*}
  R^{2-2n}\int_{P_{R}(x_{0},t_{0})}e(A(t),\phi(t))dV_g\,dt\leq \epsilon_0
\end{equation*}
holds, then for any $ \delta\in(0,\, \delta_0)$, we have
\begin{equation}\label{k1}
\sup\limits_{P_{\delta R }(x_{0},t_{0})}e(A(t),\phi(t)) \leq 16(\delta R)^{-4},
\end{equation}
and
\begin{equation}\label{k2}
\sup\limits_{P_{\delta R}(x_{0},t_{0})}|\triangledown_{A(t)}\phi(t)|_{H_0}^{2}\leq C_4,
\end{equation}
where $C_4$ is a constant depending on the geometry of $(X,\omega)$, the initial data $(A_0,\phi_0)$, $\delta_0$ and $R$.
\end{prop}

\section{Proof of theorem 1.1}

In this section, we give the detailed proof of Theorem \ref{5}.

\begin{proof}
The sufficiency is obviously from Prop \ref{115}. Here we only need to prove the necessity, i.e. there is a filtration of $E$ by $\phi$-invariant subbundles
\begin{eqnarray}\label{3.1}
0=E_{0}\subset E_{1}\subset \cdots \subset E_{l}=E ,
\end{eqnarray}
such that every quotient $Q_{i}=E_{i}/E_{i-1}$ is Higgs Hermitian flat.

Let $H(t)$ be the solution of (\ref{1}) on the Higgs bundle $(E, \overline{\partial }_{E},\phi_{0})$ with initial metric $H_{0}$, and $(A(t), \phi (t))$ be the solution of (\ref{3}) on the Hermitian bundle $(E, H_{0})$ with initial data $(A_{0}, \phi_{0})$, where $A_{0}=D_{(E, \overline{\partial }_{E}, H_{0})}$. We divide the following proof into three steps.
\paragraph{Step 1:}First, we use the $\epsilon$-regularity theorem (Prop. \ref{21}) to prove that
 a semistable Higgs bundle with $c_{1}(E)\cdot[\omega]^{n-1}=0$, $ ch_{2}(E)\cdot[\omega]^{n-2}=0$ is approximate Higgs Hermitian flat.
\par
From the Chern-Weil theory (\ref{111}), we have
\begin{equation*}
\begin{split}
 \mbox{\rm{YMH}}(A(t),\phi(t))
=&\int_X |F_{A(t)}+[\phi(t),\phi(t)^{*H_0}]|_{H_0}^{2}+2|\partial_{A(t)}\phi(t)|_{H_0}^{2} \frac{\omega^{n}}{n!}
\\=&\int_X |\sqrt{-1}\Lambda_{\omega}(F_{A(t)}+[\phi(t),\phi(t)^{*H_0}])|_{H_0}\frac{\omega^{n}}{n!}
\\=&\int_X |\sqrt{-1}\Lambda_{\omega}(F_{H(t)}+[\phi_0,\phi_0^{*H(t)}])|_{H(t)}\frac{\omega^{n}}{n!}.
\end{split}
\end{equation*}
 Since $\mathfrak{E}$ is Higgs semi-stable, from Proposition \ref{6} and the relations (\ref{4}),
 we obtain
\begin{equation}\label{17}
\begin{split}
 \mbox{\rm{YMH}}(A(t),\phi(t))&=\int_{X}|\sqrt{-1}\Lambda_{\omega}(F_{A(t)}+[\phi(t),\phi(t)^{*H_0}])|^{2}_{H_0}\frac{\omega^{n}}{n!}
\\&=\int_{X}|\sqrt{-1}\Lambda_{\omega}(F_{H(t)}+[\phi_{0},\phi_{0}^{*H(t)}])|_{H(t)}^{2}\frac{\omega^{n}}{n!}
\rightarrow 0.
\end{split}
\end{equation}
as $t\rightarrow \infty$.
Then for any $\epsilon > 0$, there exists a positive constant $t(\epsilon)$, such that for $ \forall t>t(\epsilon)$

$$\mbox{\rm{YMH}}(A(t),\phi(t))<\epsilon.$$

Choose
$R=\frac{i_{X}}{2} $ and $\epsilon=\epsilon_0R^{2n-4}$, where $\epsilon_0$ is the constant of Theorem \ref{21}. Then for $\forall t>t_0=t(\epsilon)+i_{X}^{2}$ and $\forall x_0\in X$, we have

\begin{equation*}
\begin{split}
&R^{2-2n}\int_{P_{R}(x_0,t)}e(A,\phi)dV_g\,d\tau
\\ \leq&R^{2-2n}\int_{t-R^{2}}^{t+R^{2}}\mbox{YMH}(A(\tau),\phi(\tau))d\tau
\\<&\epsilon R^{4-2n}=\epsilon_0.
\end{split}
\end{equation*}
By the $\epsilon $-regularity theorem (Prop.\ref{21} ), we have
\begin{eqnarray*}e(A,\phi)(x_0,t)\leq\sup\limits_{P_{\delta_0R/2}(x_{0},t)}e(A,\phi)<16(\delta_0 R/2)^{-4}=256(\delta_0 R)^{-4}
\end{eqnarray*}
and
\begin{eqnarray*}|\nabla_{A}\phi|_{H_0}^{2}(x_0,t)\leq\sup\limits_{P_{\delta_0R/2}(x_{0},t)}|\nabla_{A}\phi|_{H_0}^{2}\leq C_5,
\end{eqnarray*}
where $\delta_0$ is a constant in Theorem \ref{21}, and $C_5$ is a constant depending on the geometry of $(X,\omega)$ and the initial data $(A_0,\phi_0)$.
 In addition, $\sup\limits_X \left(|F_A+[\phi,\phi^{*H_0}]|_{H_0}^2+2|\partial_A\phi|_{H_0}^2\right)$ and $\sup\limits_X {|\nabla_{A(t)}\phi(t)|_{H_0}^2}$ are bounded in $\left[0,t_0\right]$.
Then from the Bochner type inequality (\ref{11}), we know there exists a positive constant $C_6$, such that
\begin{equation*}
\begin{split}
&(\triangle_g-\frac{\partial}{\partial t})\left(|F_A+[\phi,\phi^{*H_0}]|_{H_0}^2+2|\partial_A\phi|_{H_0}^2\right)
\\ \geq&-C_{6}\left(|F_A+[\phi,\phi^{*H_0}]|_{H_0}^2+2|\partial_A\phi|_{H_0}^2\right).
\end{split}
\end{equation*}
Using the  parabolic mean value inequality (Theorem 14.7 in \cite{PL}), we have
\begin{equation*}
\sup\limits_{X}e(A(t),\phi(t))\leq C_{7}\int_{X}
 e(A(t-1),\phi(t-1))dVg,
\end{equation*}
   where $C_{7}$ is a positive constant. Applying (\ref{17}), we get
 \begin{eqnarray}\label{k3}\sup\limits_{X} \left(|F_A+[\phi,\phi^{*H_0}]|_{H_0}^2+2|\partial_A\phi|_{H_0}^2\right)\rightarrow 0\end{eqnarray} as $t\rightarrow +\infty$.
\\
\paragraph{Step 2:}
Suppose $(E,\overline{\partial}_E,\phi)$ is strictly semi-stable. There is a Higgs coherent subsheaf $S$ of minimal rank $p$ with $\mbox{deg}(S)=0$. This implies $S$ is Higgs stable. After replacing $S$ by its double dual $S^{\ast \ast}$, we can assume that $S$ is reflexive. We show that $S$ is in fact a Higgs subbundle.

Consider the following exact sequence of Higgs sheaves:
\begin{equation}\label{114}
0\longrightarrow (S,\phi_S)\longrightarrow(E,\phi)\longrightarrow(Q,\phi_Q)\longrightarrow0,
\end{equation}
where $\phi_S$, $\phi_Q$ are the induced Higgs field.
For each $H(t)$, we have the bundle isomorphism on $X\setminus Z$, the locally free part of $S$:
$$f_H: S\oplus Q\rightarrow E.$$
From the Gauss-Codazzi equation (\ref{112}), the $(1,1)$ part of the pull-back curvature $f_H^*(F_{H,\overline{\partial},\phi})$ is
\begin{equation*}
f^*_H(F_{H, \overline{\partial }_{E}, \phi})^{1,1}=\left(\begin{matrix}
F_{H_S,\phi_S}^{1,1}-\gamma \wedge \gamma^{\ast} +\zeta \wedge \zeta^{\ast }  &  D_{S\otimes Q^{\ast}}^{1, 0}\gamma +\zeta \wedge \phi^*_{Q}+\phi_{S}^{\ast}\wedge \zeta \\
D_{S^{\ast}\otimes Q}^{0, 1}\gamma^{\ast} +\zeta^{\ast} \wedge \phi_{S}+\phi_{Q}\wedge \zeta^{\ast}  & F_{H_Q,\phi_Q}^{1,1}-\gamma^{\ast} \wedge \gamma +\zeta^{\ast} \wedge \zeta  \\
\end{matrix}
\right )
\end{equation*}
where $F_{H_S,\phi_S}^{1,1}=F_{H_S}+[\phi_S,\phi_S^*]$ and $F_{H_Q,\phi_Q}^{1,1}=F_{H_Q}+[\phi_Q,\phi_Q^*]$.
Since $\mbox{tr}([\phi_S,\phi_S^*])=0$ and
 $\mbox{tr}\frac{\sqrt{-1}}{2\pi}(-\gamma \wedge \gamma^{\ast} +\zeta \wedge \zeta^{\ast })\geq0$, we have \begin{eqnarray}\label{3.9}
 \mbox{tr}\frac{\sqrt{-1}}{2\pi}F_{H_S}\leq\mbox{tr}\frac{\sqrt{-1}}{2\pi}(F_{(H, \overline{\partial }_{E}, \phi)}^{(1,1)}|_S)\leq \varepsilon(t)\omega,
 \end{eqnarray}
 on $X\setminus Z$, where $\varepsilon(t)=C(n,r)\sup\limits_{X}|F_{(H, \overline{\partial }_{E}, \phi)}|\rightarrow 0$ as $t\rightarrow \infty.$

 Now, we consider the determinant line bundle $\det{S}=(\wedge^pS)^{**}$ and equip it with the possibly singular metric $\det{H_S(t)}$ on \X. By (\ref{3.9}), we have $\frac{\sqrt{-1}}{2\pi}\partial\overline{\partial}\log \det{H_S}+ \varepsilon(t)\omega\geq 0$ in current sense on \X. Then
\begin{equation}
\begin{split}
0&\leq \int_X\left(\frac{\sqrt{-1}}{2\pi}\partial\overline{\partial}\log \det{H_S}+ \varepsilon(t)\omega \right)\wedge \frac{\omega^{n-1}}{(n-1)!}
\\&=\mbox{deg}((\det{S})^{*})+\varepsilon (t) \int_{X}\frac{\omega^{n}}{n!}
\\&=\varepsilon (t) \int_{X}\frac{\omega^{n}}{n!}.
\end{split}
\end{equation}
So, subsequently, $\frac{\sqrt{-1}}{2\pi}\partial\overline{\partial}\log \det{H_S(t)}  $ converges to zero weakly in current sense. In particular $c_1(S)=0$ and $\det{S}$ is Hermitian flat. Furthermore, since $(E, \overline{\partial }_{E}, \phi )$ is Higgs approximate Hermitian flat,  $(\wedge^pE\otimes (\det{S})^{-1},\phi_p)$ is also Higgs approximate Hermitian flat, where $\phi_p$ is the induced Higgs field on $\wedge^pE\otimes (\det{S})^{-1}$. Prop.\ref{113} together with the fact that the bundle morphism $\det{S}\rightarrow \wedge^pE$ can be seen as a $\phi_p$-invariant section of $\wedge^pE\otimes (\det{S})^{-1}$ imply $\det{S}\rightarrow \wedge^pE$ has no zeros, i.e. the bundle morphism $\det{S}\rightarrow\wedge^pE$  is injective. Then the Lemma 1.20 in \cite{demailly1994compact} implies that $S$ is a Higgs subbundle of $E$ and (\ref{114}) is an exact sequence of Higgs bundles on \X.
\\

\paragraph{Step 3:}  It remains to show that $S$ is Higgs Hermitian flat and $Q$ is Higgs semi-stable  with $c_1(Q)\cdot[\omega]^{n-1}=ch_2(Q)\cdot[\omega]^{n-2}=0.$

From the exact sequence (\ref{114}), we have \begin{eqnarray}ch_2(S)+ch_2(Q)=ch_2(E),\ \  c_1(S)+c_1(Q)=c_1(E).\end{eqnarray} This shows that $c_1(Q)\cdot[\omega]^{n-1}=0.$
 On the other hand,  since $(S,\phi_S)$ is Higgs stable and $(E,\phi)$ is Higgs semistable, it is obviously that $(Q, \phi_{Q})$ is Higgs semi-stable.

By the Chern-Weil theory (\ref{111}) and  Proposition 2.3., we have $ch_2(S)\cdot[\omega]^{n-2}\geq 0$ and $ch_2(Q)\cdot[\omega]^{n-2}\geq 0$. Together with $ch_2(E)\cdot[\omega]^{n-2}=0$, we have $ch_2(S)\cdot[\omega]^{n-2}=ch_2(Q)\cdot[\omega]^{n-2}=0.$ Then $S$ is Higgs Hermitian flat. Since $(Q,\phi_Q)$ is a semi-stable Higgs bundle with $c_1(Q)\cdot[\omega]^{n-1}=ch_2(Q)\cdot[\omega]^{n-2}=0$, we set $(E_1,\phi_{E_1})=(S,\phi_S)$ and obtain the filtration (\ref{3.1}) by induction on the rank.

\end{proof}

\par
\bibliographystyle{plain}

\end{document}